\documentclass{article}
\usepackage[affil-it]{authblk}
\usepackage[utf8]{inputenc}
\usepackage{amsmath, amsfonts, amsthm, amssymb, latexsym}
\usepackage{geometry}
\usepackage{enumerate}
\usepackage{multirow}
\usepackage{rotating}
\usepackage{graphicx,color}
\usepackage{float}
\usepackage{mathrsfs}
\usepackage{enumitem}

\usepackage{xargs}
\usepackage[pdftex,dvipsnames]{xcolor}

\usepackage{hyperref}
\hypersetup{colorlinks=true, pdfstartview=FitV,linkcolor=blue!70!black,citecolor=red!70!black, urlcolor=green!60!black}
\definecolor{labelkey}{rgb}{0.6,0,0}

\usepackage[colorinlistoftodos,prependcaption,textsize=small]{todonotes}
\newcommandx{\change}[2][1=]{\todo[#1]{#2}}
\newcommandx{\unsure}[2][1=]{\todo[linecolor=red,backgroundcolor=red!25,bordercolor=red,#1]{#2}}
\newcommandx{\rmk}[2][1=]{\todo[linecolor=blue,backgroundcolor=blue!25,bordercolor=blue,#1]{#2}}
\newcommandx{\info}[2][1=]{\todo[linecolor=OliveGreen,backgroundcolor=OliveGreen!25,bordercolor=OliveGreen,#1]{#2}}
\newcommandx{\improvement}[2][1=]{\todo[linecolor=Plum,backgroundcolor=Plum!25,bordercolor=Plum,#1]{#2}}
\newcommandx{\thiswillnotshow}[2][1=]{\todo[disable,#1]{#2}}

\newtheorem{thm}{Theorem}[section]

\newtheorem{prop}[thm]{Proposition}

\theoremstyle{definition}

\theoremstyle{remark}

\title{On inverse problems for a strongly damped
wave equation on compact manifolds}

\author[1]{Li Li \thanks{lil33@uci.edu}}
\affil[1]{Department of Mathematics, University of California, Irvine, CA 92697, USA}

\author[2]{Yang Zhang \thanks{yangzh26@uw.edu}}
\affil[2]{Department of Mathematics, University of Washington,
Seattle, WA 98105, USA}

\date{}

\begin{document}

\maketitle

\noindent \textbf{ABSTRACT.}\,
We consider a strongly damped
wave equation on compact manifolds, both with and without boundaries, and formulate the corresponding inverse problems.
For closed manifolds, we prove that the metric can be uniquely determined, up to an isometry, from the knowledge of the source-to-solution map.
Similarly, for manifolds with boundaries,
we prove that the metric can be uniquely determined, up to an isometry, from partial knowledge of the Dirichlet-to-Neumann map.
The key point is to retrieve the spectral information of the Laplace-Beltrami operator, from the Laplace transform of the measurements.
Further we show that the metric can be determined up to an isometry, using a single measurement in both scenarios.

\section{Introduction}

Let $(M, g)$ be a smooth connected compact Riemannian manifold, with or without boundary, of dimension $n \geq 2$. We consider the strongly damped wave operator
$$\partial^2_{t}  -\Delta_g - \partial_t \Delta_g$$
where $-\Delta_g$ is the positive Laplace-Beltrami operator on $M$.
In the following, we study inverse problems related to this operator, considering manifolds both with and without boundaries.

First, for a closed Riemannian manifold $(M, g)$, we consider the initial value problem
\begin{equation}\label{stdampclosed}
\left\{
\begin{aligned}
(\partial^2_{t}  -\Delta_g - \partial_t \Delta_g) u&= f,\quad \,\,\, (x, t)\in M\times (0, \infty),\\
u(0)= \partial_t u(0)&= 0,\quad \,\,\,x\in M.\\
\end{aligned}
\right.
\end{equation}
For a given nonempty open subset $W \subset M$, we formally define the source-to-solution map
\begin{equation}\label{stosol}
L_W: f\to u_f|_{W\times (0, \infty)}, \qquad \mathrm{supp} (f)\subset W\times (0, \infty),
\end{equation}
where $u_f$ is the solution to (\ref{stdampclosed}).
We prove that $L_W$
is well-defined for sufficiently smooth $f$, see Proposition \ref{wellpose1}.
Our goal is to determine the metric $g$ up to an isometry from the knowledge of $L_W$. The following theorem is our first main result.

\begin{thm}\label{thm_L}
Suppose $(M, g)$, $(M, \tilde{g})$ are smooth connected closed Riemannian manifolds of dimension $n \geq 2$.
Let $W\subset M$ be a nonempty open subset
and let
$L_W, \tilde{L}_W$ be the  source-to-solution maps corresponding to $g, \tilde{g}$.
Suppose
\begin{equation}\label{stosoleq}
L_W f= \tilde{L}_W f,\qquad
\text{for }f\in C^2_c((0, \infty); L^2(W))\,\,\mathrm{with}\,\, \int_W f(t)= 0.
\end{equation}
Then $(M, g)$ and $(M, \tilde{g})$ are isometric.
\end{thm}
Next, for a manifold $M$ with boundary, we consider the initial boundary value problem
\begin{equation}\label{stdampwith}
\left\{
\begin{aligned}
(\partial^2_{t}  -\Delta_g - \partial_t \Delta_g) u&= 0,\quad \,\,\, (x, t)\in M\times (0, \infty),\\
u & =f, \quad \,\,\, (x, t)\in \partial M\times (0, \infty),\\
u(0)= \partial_t u(0)&= 0,\quad \,\,\,x\in M.\\
\end{aligned}
\right.
\end{equation}
For given relatively open $S_{\mathrm{in}}, S_{\text {out}} \subset \partial M$ satisfying $S_{\mathrm{in}} \cap S_{\mathrm{out}}\neq \emptyset$, we formally define the Dirichlet-to-Neumann map
\begin{equation}\label{DN}
\Lambda_S: f\to \partial_\nu u|_{S_{\text {out}}\times (0, \infty)},\qquad
\mathrm{supp}_x f\subset S_{\text {in}},
\end{equation}
where $\nu$ is the outward unit normal vector field along the boundary $\partial M$.
We prove that $\Lambda_S$
is well-defined for sufficiently smooth $f$, see Proposition \ref{wellpose2}.
Our goal is to determine the metric $g$ up to an isometry from the knowledge of $\Lambda_S$. The following theorem is our second main result.

\begin{thm}\label{thm_Lambda}
Suppose $(M, g)$, $(M, \tilde{g})$ are smooth connected compact Riemannian manifolds with boundary of dimension $n \geq 2$. Let $S_{\mathrm{in}}, S_{\text {out}} \subset \partial M$ be relatively open subsets satisfying $S_{\mathrm{in}} \cap S_{\mathrm{out}}\neq \emptyset$
and let
$\Lambda_S, \tilde{\Lambda}_S$ be the Dirichlet-to-Neumann maps corresponding to $g, \tilde{g}$.
Suppose $g= \tilde{g}$ on $\partial M$ and
\begin{equation}\label{DNeq}
\Lambda_S f= \tilde{\Lambda}_S f,\qquad
\text{for }f\in C^3_c((0, \infty); H^\frac{3}{2}(\partial M))\,\,\mathrm{with}\,\,\mathrm{supp}_x f\subset
S_{\text {in}}.
\end{equation}
Then $(M, g)$ and $(M, \tilde{g})$ are isometric.
\end{thm}

Note that in both Theorem 1.1 and Theorem 1.2,
we need infinite measurements on the entire time interval $(0, \infty)$.
Later we prove in Section \ref{sec_single} that it is possible to make a single measurement on a finite time interval to determine the isometric class of the metric, due to the time analyticity of the forward problem.
For more details, see Theorem \ref{thm_singleL} and \ref{thm_singleL}.

\subsection{Connection with earlier literature}
%
%

There is a vast literature on the metric determination problem associated with hyperbolic equations.
The approach mainly relies on the boundary control (BC) method (see \cite{belishev1992reconstruction}) and its variant.
For the study of metric determination on manifolds with boundaries, one can refer to \cite{katchalov2004equivalence,lassas2014inverse,lassas2010inverse}, while investigations on closed manifolds are detailed in \cite{helin2018correlation, krupchyk2008inverse}. Stability results related to this topic can be found in \cite{stefanov1998stability}.

%
%

In addition, several existing mathematical works address the metric determination problem associated with fractional equations.
Determination results related to time-fractional operators can be found in \cite{kian2018global, kian2021uniqueness}; results concerning space-fractional operators are available in \cite{feizmohammadi2021fractional};
and those pertaining to uncoupled space and time-fractional operators are detailed in \cite{helin2020inverse}.
It turns out that these fractional problems can be connected to the determination problem associated with the classical wave equation.

{
Our proof for the inverse problem on a manifold with boundary is based on a Borg-Levinson type inverse spectral result, see \cite{borg1946umkehrung, levinson1949inverse, nachman1988n, isozaki1989some, katchalov1998multidimensional, kachalov2001inverse, krupchyk2010borg, choulli2013stability}.
This idea is used by several authors in the context of other linear equations, for example, see \cite{canuto2001determining, kian2018global}.

%
%

We are motivated by
the Westervelt equation with a strong damping term arising in nonlinear acoustics.
The Westervelt equation is a quasilinear wave equation.
In some cases, the propagation of waves in a thermoviscous medium is described by the Westervelt equation with strongly damping, see \cite{kaltenbacher2018fundamental}.
Inverse problems of recovering a time-dependent potential and nonlinearity in the Euclidean space is recently considered in \cite{li2023westervelt}.
The operator considered in this paper is the linear part of that equation and our result can be regarded as the metric determination from the first-order linearization of measurements for nonlinear problems on compact manifolds.
Inverse problems for Westervelt equations has been considered in \cite{acosta2022nonlinear, eptaminitakis2022weakly}, for a general nonlinearity in \cite{uhlmann2023inverse},
and with various damping effects in
\cite{kaltenbacher2021identification,kaltenbacher2023simultaneous,kaltenbacher2023nonlinearity,zhang2023nonlinear,li2023inverse,fu2023inverse,kaltenbacher2022determining}.
Linear or semilinear strongly damped wave equations has been investigated by many authors, including but not limited to \cite{pata2005strongly, ghidaglia1991longtime, kawashima1992global, kalantarov2009finite,ikehata2013wave}.
There are also several works on inverse problems for other linear or nonlinear models with strongly damping, see \cite{colombo2007inverse, colombo2008identification,kim2023reconstruction}.

In this paper, we focus on the inverse problems of metric determination for a wave equation with a strong damping.
Although the representation formula for the measurement map becomes more complicated,
see Section \ref{subsec_L} and \ref{subsec_Lambda},
it turns out we are still able to retrieve spectral
information from the measurements by taking the Laplace transform.
This enables us to apply inverse spectral theory established in earlier literature to determine the isometric class of the metric.

\subsection{Organization}
The rest of this paper is organized as follows.
In Section \ref{sec_closed}, we work on closed manifolds and prove Theorem \ref{thm_L}. In Section \ref{sec_boundary}, we work on manifolds with boundaries and prove Theorem \ref{thm_Lambda}. In Section \ref{sec_single}, we prove the results for single measurement, which strengthen the main theorems, see Theorem \ref{thm_singleL} and \ref{thm_singleL}.

\subsection{Notations}
Throughout this paper, suppose $(M, g)$ is a smooth connected compact Riemannian manifold, with or without boundary, of dimension $n \geq 2$.
Let $\langle\cdot, \cdot\rangle$ denote the standard $L^2$-distributional pairing and
$H^r(M)$ denote the standard $L^2$-based Sobolev space $W^{r,2}(M)$.
We denote by
\begin{equation}\label{plusminusroot}
\lambda^{\pm}:= \frac{-\lambda\pm\sqrt{\lambda^2- 4\lambda}}{2}
\end{equation}
 the two (complex) roots of the quadratic equation
$z^2+\lambda z+ \lambda= 0.$

\medskip

\noindent \textbf{Acknowledgements.} L.L. and Y.Z. would like to thank Professor Katya Krupchyk and Professor Gunther Uhlmann for helpful discussions.

\section{On closed manifolds}\label{sec_closed}
Throughout this section, suppose $(M, g)$ is a smooth connected closed Riemannian manifold of dimension $n \geq 2$ and $W\subset M$ is nonempty open subset.
Let $0= \lambda_0< \lambda_1\leq \lambda_2<\cdots\to +\infty$ be the eigenvalues of $-\Delta_g$. Then there exists an orthonormal basis of $L^2(M)$ consisting of the eigenfunctions $\varphi_k$ corresponding to $\lambda_k$. Since the first eigenvalue of
$-\Delta_g$ on closed $M$ is always $0$, with the eigenspace consisting of constant functions. We will restrict ourselves to sources and solutions which are $L^2$-orthogonal to $1$, although this is not mandatory.

\subsection{Representation formula for source-to-solution map} \label{subsec_L}
We first show the global well-posedness of (\ref{stdampclosed}) for sufficiently regular $f$ with compact support.

\begin{prop}\label{wellpose1}
For $f\in C^2_c((0, \infty); L^2(M))$ satisfying $\langle f(t), 1\rangle = 0$, there exists a unique solution
$$u\in C^2([0, \infty); H^2(M))\cap L^\infty(0, \infty; H^2(M))$$
to (\ref{stdampclosed}) satisfying $\langle u(t), 1\rangle = 0$.
\end{prop}
\begin{proof}
We construct the solution in the series form of $\sum^\infty_{k= 1}u_k(t)\varphi_k$.
To determine $u_k$, we solve the ODE
\begin{equation}
\left\{
\begin{aligned}
(\frac{d^2}{dt^2} + \lambda_k\frac{d}{dt}+ \lambda_k) u_k&= f_k,\\
u_k(0)= \partial_t u_k(0)&= 0.\\
\end{aligned}
\right.
\end{equation}
It has the solution
$$u_k = K_k* f_k,$$
where $K_k(t)$ satisfies the Laplace transform
$$\mathcal{L}K_k(s)= \frac{1}{s^2+ \lambda_k s+ \lambda_k}.$$
More precisely,
\begin{itemize}[itemsep=-5pt,topsep=-2pt]
    \item[(i)]
    if $\lambda_k< 4$, then $\lambda_k^{\pm}= \frac{-\lambda_k\pm i \sqrt{4\lambda_k-\lambda_k^2}}{2}$ and we have
    $$K_k(t)= \frac{1}{\sqrt{\lambda_k- \lambda_k^2/4}}e^{-\frac{\lambda_k}{2}t}\sin(\sqrt{\lambda_k- \lambda_k^2/4}t);$$
    \item[(ii)] if $\lambda_k= 4$, then $\lambda_k^{\pm}= -2$ and we have
    $$K_k(t)= te^{-2t};$$
    \item[(iii)] if $\lambda_k > 4$, then $\lambda_k^{\pm}= \frac{-\lambda_k\pm \sqrt{\lambda_k^2- 4\lambda_k}}{2}$ and we have
    $$K_k(t)= \frac{1}{\lambda_k^+ - \lambda_k^-}(e^{\lambda_k^+ t}- e^{\lambda_k^- t}).$$
\end{itemize}
For $f_k(t)= \langle f(t), \varphi_k\rangle\in C^2_c((0, \infty))$, the $j^{th}$ derivatives of $u_k$ satisfies
$$u^{(j)}_k = K_k* f^{(j)}_k,\qquad j= 0, 1, 2.$$
Now suppose $\mathrm{supp}\,f\subset (0, T_0)\times M$ for some $T_0$.
For large $k$ such that $\lambda_k > 4$, by Cauchy-Schwarz inequality we have
\begin{equation}\label{ukest}
|u^{(j)}_k(t)|^2\leq \frac{1}{(\lambda_k^+ - \lambda_k^-)^2}\int^{\min\{t, T_0\}}_0|e^{\lambda_k^+ (t-\tau)}- e^{\lambda_k^- (t-\tau)}|^2\,d\tau\int^{\min\{t, T_0\}}_0|f^{(j)}_k(\tau)|^2\,d\tau
\end{equation}
$$\leq
\frac{4}{\lambda_k^2- 4\lambda_k}T_0\int^{T_0}_0|f^{(j)}_k(\tau)|^2\,d\tau.$$
By the dominated convergence theorem, we have
$$\sum^\infty_{k= N}\lambda_k^2 |u^{(j)}_k(t)|^2\leq
8T_0\int^{T_0}_0\sum^\infty_{k= N}|f^{(j)}_k(\tau)|^2\,d\tau\to 0$$
as $N\to \infty$. The uniform convergence of the series on $[0, \infty)$ ensures that
$\sum^\infty_{k= 1}u_k(t)\varphi_k$
is the solution to (\ref{stdampclosed}).
\end{proof}

Now let $\lambda_{j_k}$ be the distinct eigenvalues and $P_k$ be the projections on the corresponding eigenspaces.
We define $P_{W, k}:= (P_k\circ e)|_W$ where $e$ is the zero extension map from $L^2(W)$ to $L^2(M)$. Note that
$P_{W, k}\in B(L^2(W))$, where $B(L^2(W))$ denotes the space of bounded linear operators on $L^2(M)$.
The unique continuation property of elliptic operators ensures that $P_{W, k}$ is not the zero map.

\begin{prop}\label{LWHW}
The knowledge of $L_Wf$ for all $f\in C^2_c((0, \infty); L^2(W))$ with $\langle f(t), 1\rangle= 0$
determines the operator-valued function
\begin{equation}\label{HWs}
H_W(s):=\sum^\infty_{k=1}\frac{1}{s^2+\lambda_{j_k} s+ \lambda_{j_k} }P_{W, k},\qquad \mathrm{Re}\,s> 0.
\end{equation}
\end{prop}
\begin{proof}
We fix a nonzero non-negative $a(t)\in C^2_c((0, \infty))$ such that $\mathcal{L}a> 0$ on $(0, \infty)$. We consider a source $f$ of the form
$$f= a(t)\xi,\qquad \text{where }\xi\in L^2(W)\,\,\mathrm{with}\,\,\langle\xi, 1\rangle= 0.$$
Recall we use $u_f$ to denote the solution associated with $f$.
Then for each $s$ with $\mathrm{Re}\,s> 0$,
the function
$e^{-st}u_f(t)$ is $L^2$-valued integrable on $(0, \infty)$.
Then by Proposition 23 in \cite{helin2020inverse}, we
can take the Laplace transform componentwise to obtain
$$\mathcal{L}u_f(s)= \sum^\infty_{k=1} \mathcal{L}(K_k* f_k)(s)\varphi_k= \sum^\infty_{k=1}\frac{1}{s^2+\lambda_{k} s+ \lambda_{k} }\mathcal{L}f_k(s)\varphi_k$$
$$= \sum^\infty_{k=1}\frac{1}{s^2+\lambda_{k} s+ \lambda_{k} }\langle\mathcal{L}f(s), \varphi_k\rangle\varphi_k
=  \mathcal{L}a(s)\sum^\infty_{k=1}\frac{1}{s^2+\lambda_{j_k} s+ \lambda_{j_k} }P_k\xi,$$
based on the formula for solutions to (\ref{stdampclosed}) that we derive above.
Hence, we can conclude that $L_Wf$ determines the restriction of $\frac{1}{\mathcal{L}a(s)}\mathcal{L}u_f(s)$ in $W$, which equals to $H_W\xi$.
\end{proof}

\subsection{Proof of Theorem 1.1}\label{subsec_L2}
For the metric $g$, we consider the spectral data set
\begin{equation}\label{SD}
\mathrm{SD}(g):= \{(\lambda_{j_k}, P_{W, k}): k\in \mathbb{N}\}.
\end{equation}
To prove Theorem 1.1, it suffices to show that
$\mathrm{SD}(g)$ can be determined from the knowledge of the source-to-solution map $L_W$.
Indeed, we know the fact that $g$ can be determined up to an isometry, from the source-to-solution map associated with the classical wave equation, which is completely characterized by
$\mathrm{SD}(g)$.
More precisely, we have the following well-known result.

\begin{prop}[{\cite[Theorem 2]{helin2018correlation}}]\label{classicstosol}
The metric $g$ can be determined up to an isometry from the source-to-solution map
$$L^{\mathrm{hyp}}_W: f\to u_f|_{W\times (0, \infty)},\qquad f\in C^\infty_c((0, \infty)\times W)$$
associated with the hyperbolic equation
\begin{equation}
\left\{
\begin{aligned}
(\partial^2_{t}  -\Delta_g) u&= f,\quad \,\,\, (x, t)\in M\times (0, \infty),\\
u(0)= \partial_t u(0)&= 0,\quad \,\,\,x\in M.\\
\end{aligned}
\right.
\end{equation}
Here $L^{\mathrm{hyp}}_W$ has the representation formula
$$L^{\mathrm{hyp}}_Wf(x, t)= \sum^\infty_{k=0}\int^t_0 K^{\mathrm{hyp}}_{j_k}(t-\tau)P_{W, k}f(x, \tau)\,d\tau,$$
where $j_0= 0$ with
$$K^{\mathrm{hyp}}_0(t)= t, \qquad
K^{\mathrm{hyp}}_k(t)= \frac{\sin(\sqrt{\lambda_k}t)}{\sqrt{\lambda_k}}\quad (k\geq 1).$$
\end{prop}
Note that for each $z\neq -1$,
one has $|z^2+\lambda_{j_k} z+ \lambda_{j_k}|\to \infty$
as $k\to \infty$.
Then by the analytic continuation,
equation (\ref{HWs}) determines the
$B(L^2(W))$-valued meromorphic function $H_W(z)$ in $\mathbb{C}\setminus\{-1\}$ with poles
$\{\lambda_{j_k}^{\pm}\}$.

A useful observation is that we have different poles for different $\lambda$.
In fact, given  $\lambda > 4$, it follows that $\lambda^- < -2$ is strictly decreasing as $\lambda\to \infty$, while
$\lambda^+\in (-2, -1)$ is strictly increasing as $\lambda\to \infty$.
Hence, for
$\lambda, \tilde{\lambda}> 0$, the condition
$\lambda\neq \tilde{\lambda}$ implies
$\{\lambda^{\pm}\}\cap \{\tilde{\lambda}^{\pm}\}= \emptyset$.
Moreover, note that
$$\lim_{z\to \lambda^+_{j_k}}c_k(z- \lambda^+_{j_k})^{l_k}H_W(z)= P_{W, k},$$
where we set $c_k= 1, l_k= 2$ if $\lambda_{j_k}= 4$ and we set
$c_k= \lambda^+_{j_k}-\lambda^-_{j_k}, l_k= 1$ if $\lambda_{j_k}\neq 4$.
Now based Proposition \ref{LWHW} and Proposition \ref{classicstosol}, we complete the proof of Theorem 1.1 by proving the following proposition.
\begin{prop}\label{PfTh1}
Let ${H}_W(z), \tilde{H}_W(z)$ be the meromorphic functions
associated with $g, \tilde{g}$ respectively. Suppose $H_W (z)= \tilde{H}_W(z)$. Then we have $\mathrm{SD}(g)= \mathrm{SD}(\tilde{g})$.
\end{prop}
\begin{proof}
Based on the assumption, for each $k \in \mathbb{N}$, we have
\begin{equation}\label{HWtildeeq}
P_{W, k}= \lim_{z\to \lambda^+_{j_k}}c_k(z- \lambda^+_{j_k})^{l_k}H_W(z) =\lim_{z\to \lambda^+_{j_k}}c_k(z- \lambda^+_{j_k})^{l_k}\tilde{H}_{W}(z),
\end{equation}
The RHS of this equation
equals to $\tilde{P}_{W, l}$, provided $\lambda_{j_k}= \tilde{\lambda}_{\tilde{j}_l}$ for some $l\in \mathbb{N}$.
Conversely,
the RHS equals to $0$, provided $\lambda_{j_k}\notin
\{\tilde{\lambda}_{\tilde{j}_l}: l\in \mathbb{N}\}$. Hence (\ref{HWtildeeq}) implies
$(\lambda_{j_k}, P_{W, k})\in \mathrm{SD}(\tilde{g})$ for each $k$ and thus
$\mathrm{SD}(g)\subset \mathrm{SD}(\tilde{g})$.
We can symmetrically show that
$\mathrm{SD}(\tilde{g})\subset \mathrm{SD}(g)$.
\end{proof}

\section{On manifolds with boundary}\label{sec_boundary}
Throughout this section,
suppose $(M, g)$ is a smooth connected compact Riemannian manifold with boundary, of dimension $n \geq 2$.
Let $S_{\mathrm{in}}, S_{\text {out}} \subset \partial M$ be relatively open subsets satisfying $S_{\mathrm {in}} \cap S_{\mathrm{out}}\neq \emptyset$.
Let $0< \lambda_1\leq \lambda_2<\cdots\to +\infty$ be the eigenvalues of $-\Delta_g$ with the homogeneous Dirichlet boundary condition.
Then there exists an orthonormal basis of $L^2(M)$ consisting of the eigenfunctions $\varphi_k\in H^1_0(M)$ corresponding to $\lambda_k$.
The elliptic regularity ensures that
$\varphi_k\in C^\infty(\bar{M})$.
For convenience, we write
\begin{equation}\label{phik}
\phi_k:= \partial_\nu \varphi_k|_{\partial M}.
\end{equation}
We remark that most propositions in this section have their counterparts in Section \ref{sec_closed} although the proofs will be slightly more complicated.

\subsection{Representation formula for Dirichlet-to-Neumann map}\label{subsec_Lambda}
We use following inequality to show the convergence of series later.
\begin{prop}[{\cite[Lemma 2.3]{kian2018global}}]\label{f32est}
For $f \in H^{3 / 2}(\partial M)$, we have
$$
\sum_{k=1}^{\infty} \lambda_{k}^{-2}\left|\int_{\partial M} f(x) \phi_{k}(x) d x\right|^{2} \leqslant C\|f\|_{H^{3 / 2}(\partial M)}^{2}.
$$
\end{prop}
Now we prove the global well-posedness of (\ref{stdampwith}) for sufficiently regular $f$ with compact support. We mention that the regularity here is not optimal but this is not our main concern.

\begin{prop}\label{wellpose2}
For $f\in C^3_c((0, \infty); H^\frac{3}{2}(\partial M))$, the initial boundary value problem (\ref{stdampwith}) has a unique solution
$$u\in C^2([0, \infty); L^2(M))\cap L^\infty(0, \infty; L^2(M)).$$
\end{prop}
\begin{proof}
We construct the solution of the series form $\sum^\infty_{k= 1}u_k(t)\varphi_k$.
In this case, to determine $u_k$, we let the LHS of (\ref{stdampwith}) act on $\varphi_k$ to obtain
$$\langle (\partial^2_{t}-\Delta_g - \partial_t\Delta_g) u, \varphi_k \rangle= 0.$$
Integrating by parts (twice) and using the boundary condition, we obtain the formula
$$\langle -\Delta_g u, \varphi_k\rangle = \lambda_k
\langle u, \varphi_k\rangle+ \langle f, \phi_k\rangle_{\partial M}.$$
Then we can solve the ODE
\begin{equation}
\left\{
\begin{aligned}
(\frac{d^2}{dt^2} + \lambda_k\frac{d}{dt}+ \lambda_k) u_k&= -\frac{d}{dt}\langle f, \phi_k\rangle_{\partial M}- \langle f, \phi_k\rangle_{\partial M},\\
u_k(0)= \partial_t u_k(0)&= 0\\
\end{aligned}
\right.
\end{equation}
to obtain
$$u_k(t)= \int^t_0 K_k(t-\tau)(\int_{\partial M}f(\tau, x)\phi_k(x)\,dx)d\tau,$$
where $K_k$(t) satisfies the Laplace transform
$$\mathcal{L}K_k= -\frac{s+1}{s^2+ \lambda_k s+ \lambda_k}.$$
More precisely,
\begin{itemize}[itemsep=-5pt,topsep=-2pt]
\item[(i)] if $\lambda_k< 4$, then $\lambda_k^{\pm}= i\frac{-\lambda_k\pm \sqrt{4\lambda_k-\lambda_k^2}}{2}$ and we have
$$K_k(t)= -e^{-\frac{\lambda_k}{2}t}\cos(\sqrt{\lambda_k- \lambda_k^2/4}t)+\frac{1-\lambda_k/2}{\sqrt{\lambda_k- \lambda_k^2/4}}e^{-\frac{\lambda_k}{2}t}\sin(\sqrt{\lambda_k- \lambda_k^2/4}t);$$
\item[(ii)] if $\lambda_k= 4$, then $\lambda_k^{\pm}= -2$ and we have
$$K_k(t)= -e^{-2t}+ te^{-2t};$$
\item[(iii)] if $\lambda_k > 4$, then $\lambda_k^{\pm}= \frac{-\lambda_k\pm \sqrt{\lambda_k^2- 4\lambda_k}}{2}$ and we have
$$K_k(t)= \frac{-1-\lambda_k^+}{\lambda_k^+ - \lambda_k^-}e^{\lambda_k^+ t}
+  \frac{1+\lambda_k^-}{\lambda_k^+ - \lambda_k^-}e^{\lambda_k^- t}.$$
\end{itemize}
Now we verify the uniform convergence of the series
$\sum^\infty_{k= 1}u^{(j)}_k(t)\varphi_k$, for $j=0, 1,2$.
Suppose $\mathrm{supp}\,f\subset (0, T_0)\times M$ for some $T_0$. In fact, for large $k$ such that $\lambda_k > 4$, we can integrate by parts to obtain
$$|u^{(j)}_k(t)|\leq \frac{1}{\lambda_k^+ - \lambda_k^-}[\int^t_0(e^{\lambda_k^+ (t-\tau)}+ e^{\lambda_k^- (t-\tau)})|\int_{\partial M}\partial^{(j)}_tf(\tau, x)\phi_k(x)\,dx|d\tau $$
$$+\int^t_0(e^{\lambda_k^+ (t-\tau)}+ e^{\lambda_k^- (t-\tau)})|\int_{\partial M}\partial^{(j+1)}_tf(\tau, x)\phi_k(x)\,dx|d\tau
+ 2|\int_{\partial M}\partial^{(j)}_tf(t, x)\phi_k(x)\,dx|]$$
$$\leq \frac{1}{\lambda_k^+ - \lambda_k^-}
\int^t_0(e^{\lambda_k^+ (t-\tau)}+ e^{\lambda_k^- (t-\tau)}+ 2)(|\int_{\partial M}\partial^{(j)}_tf(\tau, x)\phi_k(x)\,dx|+
|\int_{\partial M}
\partial^{(j+1)}_tf(\tau, x)\phi_k(x)\,dx|)d\tau.$$
Then by Cauchy-Schwarz inequality, we have
$$|u^{(j)}_k(t)|^2\leq C\lambda^{-2}_k
\min\{t, T_0\}\int^{\min\{t, T_0\}}_0(|\int_{\partial M}\partial^{(j)}_tf(\tau, x)\phi_k(x)\,dx|+
|\int_{\partial M}
\partial^{(j+1)}_tf(\tau, x)\phi_k(x)\,dx|)^2 d\tau.$$
Now Proposition \ref{f32est} and the dominated convergence ensure the convergence of the series.

\end{proof}

Now let $\lambda_{j_k}$ be the distinct eigenvalues, $m_k$ be the corresponding multiplicities, and $\varphi_{j_k, p}, 1\leq p\leq m_k$ be the corresponding eigenfunctions. We define
\begin{equation}\label{Phikxy}
\Phi_k(x, y):= \sum^{m_k}_{p=1}\phi_{j_k, p}(x)\phi_{j_k, p}(y),\qquad (x,y)\in S_{\mathrm{in}}\times S_{\mathrm{out}}.
\end{equation}
Note that \cite[Lemma 2.2]{canuto2001determining} ensures that $\Phi_k$
can not be identically zero on any relatively open subset of $S_{\mathrm{out}}\times S_{\mathrm{in}}$.
As in Subsection \ref{subsec_L}, we fix a nonzero non-negative $a(t)\in C^3_c((0, \infty))$ such that $\mathcal{L}a> 0$ on $(0, \infty)$.
We consider
$$f= a(t)\xi,\qquad \text{where }\xi\in H^{3/2}(\partial M)\,\,\mathrm{with}\,\,\mathrm{supp}\,\xi\subset S_{\mathrm{in}}.$$
Recall we use $u_f$ to denote the solution associated with $f$.
As before, for $s$ with $\mathrm{Re}\,s> 0$, we can
take the Laplace transform to obtain
$$(\mathcal{L}u_f)(s)= -(\mathcal{L}a)(s)\sum^\infty_{k=1}\frac{s+1}{s^2+\lambda_{k} s+ \lambda_{k} }\int_{\partial M}\xi(x)\phi_k(x)\,dx\,\varphi_k,
$$
$$\frac{1}{s^2+2s}(\mathcal{L}u_f)'(s)= (\mathcal{L}a)(s)\sum^\infty_{k=1}\frac{1}{(s^2+\lambda_{k} s+ \lambda_{k})^2}\int_{\partial M}\xi(x)\phi_k(x)\,dx\,\varphi_k
$$
based on the formula for the solution of (\ref{stdampwith}) that we derive above.
Hence,
by considering $\mathcal{L}(\Lambda_S f)$,
we have the following proposition.
\begin{prop}\label{DNHS}
The knowledge of $\Lambda_S f$ for all $f\in C^3_c((0, \infty); H^\frac{3}{2}(\partial M))$ with $\mathrm{supp}_x f\subset
S_{\text {in}}$
determines
\begin{equation}\label{HSs}
H_S(s):= \sum^\infty_{k=1}\frac{1}{(s^2+\lambda_{k} s+ \lambda_{k})^2}\int_{\partial M}\xi(y)\Phi_k(x, y)\,dy,\qquad \mathrm{Re}\,s> 0,\,\,
x\in S_{\text {out}}
\end{equation}
for each $\xi\in H^{3/2}(\partial M)$ with $\mathrm{supp}\,\xi\subset S_{\mathrm{in}}$.
\end{prop}

\subsection{Proof of Theorem 1.2}
For the metric $g$, we consider the (Neumann) boundary spectral data set
\begin{equation}\label{BSD}
\mathrm{BSD}(g):= \{(\lambda_k, \phi_k|_{S_{\text {in}} \cup S_{\mathrm{out}}}): k\in\mathbb{N}\}.
\end{equation}
To prove Theorem \ref{thm_Lambda}, it suffices to show that the set
$\{(\lambda_{j_k}, \Phi_k): k\in\mathbb{N}\}$ can be determined from the knowledge of the Dirchlet-to-Neumann map $\Lambda_S$, since we have the following known results.

\begin{prop}\label{phitoBSD}
The set $\{(\lambda_{j_k}, \Phi_k): k\in\mathbb{N}\}$ determines $\mathrm{BSD}(g)$ up to an orthogonal transformation on $(\phi_{j_k, 1},\cdots,\phi_{j_k, m_k})$ for each $k$.
\end{prop}

\begin{prop}\label{BSDtog}
The set $\mathrm{BSD}(g)$, up to an orthogonal transformation on $(\phi_{j_k, 1},\cdots,\phi_{j_k, m_k})$ for each $k$, determines $g$ up to an isometry.
\end{prop}

We refer readers to the proof of Theorem 1.1, after equation (3.13), in \cite{canuto2001determining} for a detailed discussion on Proposition \ref{phitoBSD}. Additionally, for insights into Proposition \ref{BSDtog},
we guide readers to refer to Theorem 4.33 in \cite{kachalov2001inverse} and the subsequent remarks post equation (5.2) in \cite{kian2018global}.

Note that for each $z\neq -1$, we have $\lambda_k/(z^2+\lambda_{k} z+ \lambda_{k})^2\to 1/\lambda_k$
as $k\to \infty$.
Hence, given the convergence of $\sum^\infty_{k=1} \lambda_k b_k \varphi_k$ in $L^2(M)$ implying the convergence of $\sum^\infty_{k=1}  b_k \varphi_k$ in $H^2(M)$,  and considering the boundedness of map $u\to \partial_\nu u|_{\partial M}$ from $H^2(M)$ to $H^{1/2}(\partial M)$,  Proposition \ref{f32est}
along with the analytic continuation ensure that (\ref{HSs}) gives a
$H^{1/2}(S_{\mathrm{out}})$-valued meromorphic function
$H_S(z)$, in $\mathbb{C}\setminus\{-1\}$ with poles
$\{\lambda_{j_k}^{\pm}\}$, for each fixed $\xi$.

As in Subsection \ref{subsec_L2}, we have different poles for different $\lambda$.
Moreover, note that
\begin{equation}\label{ckHS}
\lim_{z\to \lambda^+_{j_k}}c_k(z- \lambda^+_{j_k})^{l_k}H_S(z)= \int_{\partial M}\xi(y)\Phi_k(x, y)\,dy,
\end{equation}
where we set $c_k= 1$, $l_k= 4$ if $\lambda_{j_k}= 4$, and we set
$c_k= (\lambda^+_{j_k}-\lambda^-_{j_k})^2$, $l_k= 2$ if $\lambda_{j_k}\neq 4$.
Now based Proposition \ref{phitoBSD} and Proposition \ref{BSDtog}, we complete the proof of Theorem \ref{thm_singleLambda} by proving the following proposition.
\begin{prop}\label{PfTh2}
Let ${H}_S(z), \tilde{H}_S(z)$ be the meromorphic functions
and $\Phi_k, \tilde{\Phi}_k$ be defined in (\ref{Phikxy}), corresponding to $g, \tilde{g}$ respectively.
Suppose $H_S(z)= \tilde{H}_S(z)$ for each $\xi\in H^{3/2}(\partial M)$ with $\mathrm{supp}\,\xi\subset S_{\mathrm{in}}$.
Then we have
\begin{equation}\label{lamphieq}
\{(\lambda_{j_k}, \Phi_k): k\in\mathbb{N}\}=
\{(\tilde{\lambda}_{\tilde{j}_k}, \tilde{\Phi}_k): k\in\mathbb{N}\}.
\end{equation}
\end{prop}

\begin{proof}
We can use the same argument as in the proof of
Proposition \ref{PfTh1}  to show
$$\{(\lambda_{j_k}, \int_{\partial M}\xi(y)\Phi_k(x, y)\,dy): k\in\mathbb{N}\}
= \{(\tilde{\lambda}_{\tilde{j}_k}, \int_{\partial M}\xi(y)\tilde{\Phi}_k(x, y)\,dy): k\in\mathbb{N}\}$$
based on (\ref{ckHS}),
for each $\xi\in H^{3/2}(\partial M)$ with $\mathrm{supp}\,\xi\subset S_{\mathrm{in}}$, which
implies (\ref{lamphieq}).
\end{proof}

\section{Single measurement}\label{sec_single}
In this section, we construct an appropriate function $h$, through which we can determine the metric $g$ up to an isometry from the single measurement $L_W h$ or $\Lambda_S h$. This is possible due to the convolution form and the time analyticity of the solution of (\ref{stdampclosed}) or (\ref{stdampwith}).

First, we consider the case of a closed manifold $(M,g)$ as discussed in Section \ref{sec_closed}.
The following proposition is an analogue of \cite[Proposition 17]{helin2020inverse}, which suggests that we can measure on a finite time interval instead of the whole $(0, \infty)$.

\begin{prop}\label{Ttoinfty}
For $f\in C^2_c((0, T); L^2(W))$ with $\langle f(t), 1\rangle = 0$, the knowledge of
$L_W f$ on $(0, T)$ determines $L_W f$ on $(0, \infty)$.
\end{prop}
\begin{proof}
We choose small $\epsilon> 0$ such that $\mathrm{supp}_t\,f\subset (0, T-\epsilon)$. By Proposition \ref{wellpose1}, we have
$$L_W f= \sum^\infty_{k=1}G_k(t)\varphi_k|_W,
\qquad G_k(t):=\int^T_0 K_k(t-\tau)f_k(\tau)\,d\tau,$$
for $t> T-\epsilon$.
Here $K_k, f_k$ are the same as in the proof of Proposition \ref{wellpose1}.
Observe that the  analyticity of $K_k(z)$ implies the analyticity of
$G_k(z)$.
Moreover, note that for sufficiently large $k$ with $\lambda_k > 4$, we have $\lambda^{\pm}_k< 0$.
Then for $z$ with
$\mathrm{Re}\,z> T-\epsilon$ and $0< \tau< T-\epsilon$, we have
$$|e^{\lambda^{\pm}_k(z-\tau)}|\leq 1.$$
It follows that we can estimate as in (\ref{ukest}) to show the uniform convergence of
$\sum^\infty_{k=1}G_k(z)\varphi_k|_W$ on $\{z: \mathrm{Re}\,z> T-\epsilon\}$. Hence $L_W f(z)$ is $L^2(W)$-valued analytic on $\{z: \mathrm{Re}\,z> T-\epsilon\}$. In particular, the analytic continuation implies that $L_W f$ on $(T-\epsilon, T)$ determines $L_W f$ on $(T-\epsilon, \infty)$.
\end{proof}

Now we choose $\psi_k$ such that $\mathrm{span}\,\{\psi_k\}$ is dense in the orthogonal supplement of $\mathrm{span}\,\{1\}$ in  $L^2(W)$. We define the function
\begin{equation}\label{SMh}
h:= \sum^\infty_{k=1}h_k(t)\psi_k.
\end{equation}
Here the nonzero non-negative function $h_k\in C^\infty_c((0, \infty))$ satisfies that
$$\mathrm{supp}\,h_k\subset (t_{k-1}, t_k),$$
where $0= t_0< t_1<\cdots$ and $\lim t_k\leq T$ for some constant $T> 0$.
We also require the series in (\ref{SMh}) converges uniformly in $t\in \mathbb{R}$, and it defines
$h\in C^\infty_c((0, T); L^2(W))$. An explicit construction of $h_k$ can be found in \cite[Section 4.1]{helin2020inverse} (see Figure 1 and Proposition 15 there).
Next, we prove the following proposition.

\begin{prop}\label{LWhphi}
The knowledge of $L_W h$ on $(0, T)$ determines
$L_W (h_k\psi_k)$ on $(0, \infty)$ for each $k$.
\end{prop}
\begin{proof}
We prove the statement by induction. Suppose it holds for $1,2,\cdots, j$. Note that $h_k= 0$ on
$(0, t_{j+1})$ for $k> j+1$, so
$\sum^\infty_{k= j+1}h_k\psi_k= 0$ on
$(0, t_{j+1})$ and the convolution formula for $L_W$ implies $L_W(\sum^\infty_{k= j+1}h_k\psi_k)= 0$ on $(0, t_{j+1})$. Hence we have
$$L_W h= L_W(\sum^{j+1}_{k= 1}h_k\psi_k)$$
on $(0, t_{j+1})$. The induction hypothesis implies
$L_W h$ on $(0, T)$ determines $L_W(h_{j+1}\psi_{j+1})$ on $(0, t_{j+1})$. Since
$\mathrm{supp}\,h_{j+1}\subset (t_{j}, t_{j+1})$,
this further determines $L_W(h_{j+1}\psi_{j+1})$ on $(0, \infty)$  by Proposition \ref{Ttoinfty}.
\end{proof}

Based on Proposition \ref{LWhphi}, the knowledge of $L_W h$ on $(0, T)$ determines
\begin{equation}\label{HWP}
\sum^\infty_{k=1}\frac{1}{s^2+\lambda_{j_k} s+ \lambda_{j_k} }P_k\xi,\qquad \mathrm{Re}\,s> 0
\end{equation}
for $\xi= \psi_k$ in the proof of
Proposition \ref{LWHW}. Then (\ref{HWP}) is determined for a general $\xi\in L^2(W)$ with $\langle\xi, 1\rangle= 0$ by the denseness of $\mathrm{span}\,\{\psi_k\}$. Hence Theorem 1.1 can be strengthened into the following single measurement result.
\begin{thm}\label{thm_singleL}
Suppose $(M, g)$, $(M, \tilde{g})$ are smooth connected closed Riemannian manifolds of dimension $n \geq 2$.
Let $W\subset M$ a nonempty open subset
and let
$L_W, \tilde{L}_W$ be the source-to-solution maps corresponding to $g, \tilde{g}$.
Suppose
\begin{equation}
L_W h= \tilde{L}_W h
\end{equation}
on $(0, T)$ for $h$ defined in (\ref{SMh}). Then $(M, g)$ and $(M, \tilde{g})$ are isometric.
\end{thm}

For the manifold $M$ with boundary in Section \ref{sec_boundary}, we can consider a similar $h$ of the form (\ref{SMh}),
where we choose $\psi_k$ such that $\mathrm{span}\,\{\psi_k\}$ is dense in $H^{3/2}(S_{\mathrm{in}})$ instead.
Then we can prove the following single measurement version of Theorem \ref{thm_Lambda} in the same way.

\begin{thm}\label{thm_singleLambda}
Suppose $(M, g)$, $(M, \tilde{g})$ are smooth connected compact Riemannian manifolds with boundary of dimension $n \geq 2$.
Let $S_{\mathrm{in}}, S_{\text {out}} \subset \partial M$ be relatively open subsets satisfying $S_{\text {in}} \cap S_{\mathrm{out}}\neq \emptyset$.
Let
$\Lambda_S, \tilde{\Lambda}_S$ be the Dirichlet-to-Neumann maps corresponding to $g, \tilde{g}$.
Suppose $g= \tilde{g}$ on $\partial M$ and
\begin{equation}
\Lambda_S h= \tilde{\Lambda}_S h
\end{equation}
on $(0, T)$ for $h$ defined in form of (\ref{SMh}) as described above. Then $(M, g)$ and $(M, \tilde{g})$ are isometric.
\end{thm}

\bibliographystyle{plain}
{\small\bibliography{Reference10}}
\end{document}